\documentclass[12pt]{amsart}
 \font\tenmsb=msbm10 scaled\magstep
1
\newfam\msbfam
   \textfont\msbfam=\tenmsb

\usepackage{amsmath}
\usepackage{amssymb}
\usepackage{mathscinet}
\usepackage{amsfonts,amsmath}
\usepackage{epsfig}
\usepackage{amsthm}
\usepackage{latexsym}
\usepackage[latin1]{inputenc}
\usepackage{verbatim}
\usepackage{pb-diagram}
\usepackage[labelformat=empty]{subfig}
\usepackage{everysel, keyval, ragged2e}
\usepackage{wrapfig}
\usepackage{color}

\newcommand{\captionfonts}{\tiny}

\makeatletter  
\long\def\@makecaption#1#2{%
  \vskip\abovecaptionskip
  \sbox\@tempboxa{{\captionfonts #1: #2}}%
  \ifdim \wd\@tempboxa >\hsize
    {\captionfonts #1: #2\par}
  \else
    \hbox to\hsize{\hfil\box\@tempboxa\hfil}%
  \fi
  \vskip\belowcaptionskip}
\makeatother   

\newtheorem{theorem}{Theorem}[section]
\newtheorem{lemma}[theorem]{Lemma}

\theoremstyle{definition}
\newtheorem{definition}[theorem]{Definition}

\theoremstyle{remark}
\newtheorem{remark}[theorem]{Remark}

\newtheorem{observation}[theorem]{Observation}
\numberwithin{equation}{section}

\def\R{\mathbb R}    

\begin{document}

\title[Trumping and Power Majorization]{Trumping and Power Majorization}

\author[D.~W.\ Kribs]{David W. Kribs$^{1,2}$}
\author[R.\ Pereira]{Rajesh Pereira$^1$}
\author[S.\ Plosker]{Sarah Plosker$^1$}
\address{$^1$Department of Mathematics \& Statistics, University of Guelph, Guelph, ON, Canada N1G 2W1}
\address{$^2$Institute for Quantum Computing, University of Waterloo, Waterloo, ON, Canada N2L 3G1}


\begin{abstract}
Majorization is a basic concept in matrix theory that has found applications in numerous settings over the past century. Power majorization is a more specialized notion that has been studied in the theory of inequalities. On the other hand, the trumping relation has recently been considered in quantum information, specifically in entanglement theory. We explore the connections between trumping and power majorization. We prove an analogue of Rado's theorem for power majorization and consider a number of examples.
\end{abstract}

\maketitle

\section{Introduction}
Majorization is one of the most basic concepts in matrix theory, first considered over a century ago as a way to address a diverse set of problems from economics, engineering, and physics \cite[Chapter 1]{MOA11}. More recently, it has become a central mathematical tool in quantum information theory, beginning with work of Nielsen \cite{Nie99} that linked it with quantum operations described by local operations and classical communication \cite{LoPr01}. This has enabled the use of extensive results on majorization to gain further insight into comparisons and measures of quantum entanglement, one of the most useful properties of quantum information.

More refined notions of majorization have also been studied in matrix theory, though less so in quantum information theory. With regard to \emph{power majorization}, for instance, Clausing \cite{Cla84} initially asked if it is equivalent to majorization. A counterexample for dimension $n\geq 42$ was given in \cite{Ben86}, and  for $n\geq 4$ in \cite{All88}. As stated in \cite{Ben11}, these examples were rather artificial, and it was some years before the first natural counterexamples appeared.

From a different direction, motivated initially by issues in quantum entanglement theory, Jonathan and Plenio \cite{JoPl99} introduced the idea of \emph{trumping}. 
Trumping is a generalization of majorization in that, given two incomparable vectors $x$ and $y$ (incomparable in the sense that one vector does not majorize the other), it is sometimes possible to find a third vector $z$, having all positive components, such that $x\otimes z$ is majorized by $y\otimes z$, thus allowing us to compare a larger set of vectors than was possible using majorization alone. This definition is made precise in section \ref{sec:defnsnotn}.

In this paper we prove some results that connect trumping to power majorization and we consider several examples. This provides another link between studies in matrix theory and quantum information theory. The next section covers requisite background material and results. The following two sections make the explicit geometric connections between the notions and discuss examples of trumping and power majorization.

\section{Background and Preliminary Results}\label{sec:defnsnotn}

Let $x=(x_{1},x_{2},...,x_{d})\in \mathbb{R}^d$. Then we let the vector $x^{\downarrow}=(x^{\downarrow}_{1},x^{\downarrow}_{2},...,x^{\downarrow}_{d})$ denote the vector in $\mathbb{R}^d$ whose elements consist of the elements of $x$ reordered so that $x^{\downarrow}_{1}\ge x^{\downarrow}_{2}\ge x^{\downarrow}_{3}\ge...\ge x^{\downarrow}_{d}$.  Similarly $x^{\uparrow}=(x^{\uparrow}_{1},x^{\uparrow}_{2},...,x^{\uparrow}_{d})$ denotes the vector in $\mathbb{R}^d$ whose elements consist of the elements of $x$ reordered so that $x^{\uparrow}_{1}\le x^{\uparrow}_{2}\le x^{\uparrow}_{3}\le...\le x^{\uparrow}_{d}$.

\begin{definition}\label{defn:maj}  Let $(x_{1},x_{2},...,x_{d})$ and $(y_{1},y_{2},...,y_{d})$ be two $d$-tuples of real numbers.  
Then we say that $(x_{1},x_{2},...,x_{d})$ is majorized  by $(y_{1},y_{2},...,y_{d})$, written $x\prec y$,   if 
\begin{eqnarray*}
\sum_{j=1}^{k}x^{\downarrow}_{j}\leq \sum_{j=1}^{k}y^{\downarrow}_{j}\quad 1\leq k \leq d,
\end{eqnarray*}
 and 
 \begin{eqnarray}\label{eq:majeq}
 \sum_{j=1}^{d}x_{j}= \sum_{j=1}^{d}y_{j}.
 \end{eqnarray} 
\end{definition}


The majorization relation is a partial order when restricted to the set of vectors $x\in \R^d$ that are non-increasing, with $x \prec y$ and $y \prec x$ if and only if
$x^{\downarrow} = y^{\downarrow}$. If equality does not necessarily hold in equation (\ref{eq:majeq}), we say that $x$ is \emph{sub-majorized} by $y$ and we write $x\prec_w y$, where the $w$ stands for ``weak''.

A similar definition holds if we order the components of the vectors in \emph{non-decreasing} order:
 $x$ is majorized by $y$
if 
\begin{eqnarray*}
\sum_{j=1}^k x^{\uparrow}_j \geq \sum_{j=1}^k y^{\uparrow}_j\quad 1\leq k \leq d,
\end{eqnarray*}
and 
\begin{eqnarray}\label{eq:majeq2}
\sum_{j=1}^{d}x_{j}= \sum_{j=1}^{d}y_{j}.
\end{eqnarray}
  If equality does not necessarily hold in equation (\ref{eq:majeq2}), we say that $x$ is \emph{super-majorized} by $y$ and we write $x\prec^w y$. Note that sub-majorization and super-majorization are not equivalent in general.


It is often the case that two vectors $x$ and $y$ are incomparable: that is, neither is majorized by the other. However, also motivated by entanglement theory, in \cite{JoPl99} the authors demonstrate that it is sometimes possible to find a unit vector $z$, referred to as a ``catalyst'',  such that $x\nprec y$ but $x\otimes z\prec y\otimes z$. Formally, we have the following definition.

\begin{definition}
Let $(x_{1},x_{2},...,x_{d})$ and $(y_{1},y_{2},...,y_{d})$ be two $d$-tuples of real numbers. 
We say that $x$ is \emph{trumped} by $y$ and write $x\prec_T y$ if there exists a unit vector $z\in \R^n$ with positive components such that $x\otimes z\prec y\otimes z$.  
\end{definition}

Similarly, sub-trumping (written $x\prec_{wT} y\Leftrightarrow x\otimes z \prec_w y\otimes z$) and super-trumping ($x\prec^w_T y$) can be defined. The vector  $z$ is typically taken to have all positive components to  avoid the potential issue of its components summing to zero.

In \cite{JoPl99} the authors refer to trumping as \emph{entanglement-assisted local transformation (ELQCC)}. Because ELQCC can be seen as entanglement-assisted local operations and classical communication, some choose to abbreviate this as ELOCC (see, e.g., \cite{AuNe08}). Also in \cite{AuNe08}, the authors call trumping \emph{catalytic majorization}, thereby emphasizing the link with majorization, and the role of the catalyst $z$. Moreover,  it was shown in \cite{DaKl01} that the dimension of the  catalyst may be arbitrarily large.
If we consider the vector $y$ to be infinite-dimensional by appending 0's, then the closure in the $\ell_1$ norm of the set  of all vectors $x$ trumped by $y$ (where $x$ may be infinite-dimensional but having finite support) is characterized by $||x||_p\leq||y||_p$ for all $p\geq 1$ (see \cite{AuNe08} for details).

Note that if we are comparing two vectors $x$ and $y$ to see if one is majorized or trumped by the other, we can ``delete'' corresponding zeros from both vectors without affecting the relation. In other words, we can assume without loss of generality that one of the vectors has no zero components. Similarly, if we are comparing two vectors of different dimensions, we can append zeros to the smaller vector to obtain vectors of equal dimension, and so we can effectively compare two vectors of different sizes.

Let us define $\sigma(x)=-\sum_{i=1}^dx_i\log x_i$, which we recall is the formula for the von Neumann entropy of a density matrix (that is, a positive semi-definite trace one matrix) with eigenvalues $x_i$. We use the convention that $0\log 0\equiv 0$. Let us also define $A_\nu(x)=\left(\frac1d\sum_{i=1}^dx_i^\nu\right)^{\frac1\nu}$ for real numbers $\nu\neq 0$ and $A_0(x)=\left(\prod_{i=1}^d x_i\right)^\frac{1}{d}$. In \cite{Tur07}, Turgut established the following result.

\begin{theorem}\cite{Tur07} \label{turg}
For two real $d$-dimensional vectors $x$ and $y$ with non-negative components such that $x$ has non-zero elements and the vectors are distinct up to permutation (i.e.\ $x^{\uparrow}\neq y^{\uparrow}$), the relation $x\prec_T y$ is equivalent to the following three strict inequalities:
\begin{enumerate}
\item $A_\nu(x)>A_\nu(y),\quad \forall \nu\in (-\infty, 1)$,\label{T1}
\item $A_\nu(x)<A_\nu(y),\quad \forall \nu\in (1,\infty)$,\label{T2}
\item $\sigma(x)>\sigma(y)$.
\end{enumerate}
\end{theorem}

\begin{remark}
Note that for dimensions two and three, trumping is equivalent to majorization: Clearly majorization implies trumping, and for the reverse direction, first note that for any dimension $d$, $x\prec_T y$ implies that we have $x_1^\uparrow\geq y_1^\uparrow$ and $x_d^\uparrow\leq y_d^\uparrow$. Additionally, for an $n$-dimensional catalyst $z$, trumping implies
\begin{eqnarray*}
\sum_{i, j=1}^{d,n}x_iz_j&=&\sum_{i, j=1}^{d,n}y_iz_j\\
\sum_{i=1}^{d}x_i\left(\sum_{j=1}^{n}z_j\right)&=&\sum_{i=1}^{d}y_i\left(\sum_{j=1}^{n}z_j\right)\\
\sum_{i=1}^{d}x_i&=&\sum_{i=1}^{d}y_i.
\end{eqnarray*}

The two-dimensional case becomes obvious. For $d=3$, by combining $x_3^\uparrow\leq y_3^\uparrow$ with
$x_1^\uparrow+x_2^\uparrow+x_3^\uparrow=y_1^\uparrow+y_2^\uparrow+y_3^\uparrow$, we obtain
$x_1^\uparrow+x_2^\uparrow\geq y_1^\uparrow+y_2^\uparrow$, and the result follows.

The above holds similarly for sub- and super-trumping implying sub- and super-majorization, respectively, for $d=2$.
\end{remark}

In \cite{Kli04,Kli07}, Klimesh establishes a theorem that shows trumping is equivalent to a series of inequalities for a family of additive Schur-convex functions. For a $d$-dimensional probability vector $x$, let
\[
  f_r(x) =
  \begin{cases}
    \ln \sum_{i=1}^d x_i^r & (r>1); \\
    \sum_{i=1}^d x_i \ln x_i & (r = 1); \\
    -\ln \sum_{i=1}^d x_i^r & (0<r<1); \\
    -\sum_{i=1}^d \ln x_i & (r = 0); \\
    \ln \sum_{i=1}^d x_i^r & (r<0).
  \end{cases}
\]

If any of the components of $x$ are
$0$, we take $f_r(x) = \infty$ for $r \leq 0$.

\begin{theorem}(\cite{Kli04,Kli07})
 Let $x=(x_1,\ldots,x_d)$ and $y=(y_1,\ldots,y_d)$ be $d$-dimensional
 probability vectors.
 Suppose that $x$ and $y$ do not both contain components equal to $0$
 and that $x^{\uparrow} \neq y^{\uparrow}$.  Then $x \prec_T y$ if and only if
 $f_r(x) < f_r(y)$ for all real numbers $r$.
\end{theorem}

Klimesh and Turgut's conditions are easily seen  to be equivalent.  We will use Klimesh's notation as it is more convenient for our purposes.
We now introduce the concept of power majorization which has been studied extensively.

\begin{definition}
Let $x$ and $y$ be vectors of non-negative components. We say that $x$ is \emph{power majorized} by $y$, denoted $x\preceq_p y$, if $x_1^p+\cdots +x_d^p\leq y_1^p+\cdots +y_d^p$ for all $p\geq 1, p\leq 0$ and the inequality switches direction when $0\leq p\leq 1$. In particular, we note that equality holds when $p=0, 1$. 
We define \emph{strict power majorization}, denoted $x\prec_{p}y$, to be power majorization with strict inequality, and equality if and only if $p=0,1$.
\end{definition}

Power majorization is unfortunately not as well-behaved of a partial order on vectors as majorization is, in the following sense: If $x$ is majorized by $y$ and $\sum_{i=1}^d\phi(x_i)=\sum_{i=1}^d\phi(y_i)$ for some strictly convex function $\phi$, then $x^{\uparrow}= y^{\uparrow}$. This is not the case with power majorization. Indeed, \cite[Theorem 2]{Ben11} gives an example of vectors $x$ and $y$, where $x\preceq_p y$, and  $\sum_{i=1}^d\phi(x_i)=\sum_{i=1}^d\phi(y_i)$ for $\phi(c)=c^2$, but $x^{\uparrow}\neq y^{\uparrow}$.

We note that power majorization can be expressed in terms of Klimesh's functionals.

\begin{observation} Let $x$ and $y$ be vectors in $\mathbb{R}^{d}$ with positive components.  Then $x\preceq_p y$ if and only if $f_r(x)\le f_r(y)$ for all $r\in \mathbb{R}$.   \end{observation}

It follows immediately from the definition of power majorization that $x \preceq_p y$ implies that $f_r(x)\le f_r(y)$ whenever $r\neq 0,1$.  We note that if $x\preceq_p y$ then $g(r)=\sum_{i=1}^{d}x_i^r-\sum_{i=1}^dy_i^r$ is non-negative on $(0,1)$ and non-positive on $(-\infty , 0)\cup (1,\infty )$. Since $g$ is differentiable with $g'(r)=\sum_{i=1}^{d}x_i^r \ln(x_i)-\sum_{i=1}^dy_i^r \ln(y_i)$, by evaluating $\lim_{r\rightarrow 0}\frac{g(r)}{r}$ and  $\lim_{r\rightarrow 1}\frac{g(r)}{r-1}$, we get that $g'(0)=-f_0(x)+f_0(y)\ge 0$ and $g'(1)=f_1(x)-f_1(y)\le 0$, respectively. The converse, namely $f_r(x)\le f_r(y)$ for all $r\in \mathbb{R}$ implies $x\preceq_p y$, is immediate.

We note that if $x$ is strictly power-majorized by $y$, we will have $f_0 (x)\le f_0(y)$ and $f_1 (x)\le f_1(y)$ but these inequalities may not be strict. An example of this is given by Turgut in \cite{Tur07} to show that the third inequality in Theorem \ref{turg} does not follow from the other inequalities.  The following observation now easily follows; it is especially useful for proving trumping relations between $d$-tuples of integers.

\begin{observation} \label{fact} Let $x$ and $y$ be vectors in $\mathbb{R}^{d}$ with positive components with $x\prec_p y$, then $x\prec_T y$ provided that $\prod_{i=1}^{d}x_i \neq \prod_{i=1}^d y_i$ and  $\prod_{i=1}^{d}x_i^{x_i} \neq \prod_{i=1}^d y_i^{y_i}$. \end{observation}

\section{Geometry of Trumping and Power Majorization}

We follow Daftuar and Klimesh \cite{DaKl01}, in using the notation $S(y)=\{ x\in (0,\infty)^d:x\prec y\} $ and $T(y)=\{ x\in (0,\infty)^d:x\prec_T y\} $.  Similarly we let $P(y)=\{ x\in (0,\infty)^d:x\preceq_p y \}$.  While the geometric properties of $S(y)$ are quite well-known; the properties of $T(y)$ are less known (although several interesting properties of $T(y)$ were found in \cite{DaKl01}) and even less known are  the properties of $P(y)$.  In this section we will study the geometric relationship between $T(y)$ and $P(y)$.  It is clear that $S(y) \subseteq T(y) \subseteq P(y)$.  We begin with the following closure relation.

\begin{theorem} Let $y$ be a $d$-vector all of whose components are positive. Then the set $P(y)$ is the closure in $\mathbb{R}^d$ of the set $T(y)$.   \end{theorem}

\begin{proof}  
Since all of the functions $f_r(x)$ are continuous on $(0,\infty)^d$, we have $\overline{T(y)}\subseteq P(y)$.  Now suppose $x\in P(y)$.  If all the entries of $x$ are the same, then $x\in S(y)\subseteq T(y)$, otherwise, there exists a vector $x^{\prime}\neq x$ where $x^{\prime}$ is some permutation of $x$.  Now let $z(t)=tx+(1-t)x^{\prime}$.  The function $f_r$ is either strictly convex or is the logarithm of a strictly convex function; hence, if $t\in (0,1)$, we have $f_r(z(t))<f_r(x)\le f_r(y)$, so $z(t)\in T(y)$. As $x=\lim_{t\to 0^+}z(t)$, we have $x\in \overline{T(y)}$. 
\end{proof}

It is straightforward to check that $T(y)$ is convex; this was mentioned in \cite{DaKl01}. Since $P(y)$ is the closure of  $T(y)$, it follows that  $P(y)$ is a convex set. Thus the set $P(y)$ is a closed convex set, and so it is the convex hull of its extreme points.   We recall that Rado's theorem for majorization states that $x$ is majorized by $y$ if and only if $(x_1, \dots, x_d)$ is contained in the convex hull of $(y_{\sigma(1)}, \dots, y_{\sigma(d)})$, where $\sigma$ is any permutation on $d$ elements. We will derive an analogue of Rado's theorem for power majorization.
 We first need the following result from Daftuar and Klimesh:

\begin{lemma} \cite{DaKl01} Let $x$ and $y$ be $d$-vectors all of whose components are positive.  Let $x_{max}$, $x_{min}$, $y_{max}$, $y_{min}$ be the values of the maximum and minimum entries of $x$ and $y$ respectively.  Suppose $x$ is a boundary point of $T(y)$, then either $x_{max}=y_{max}$ or $x_{min}=y_{min}$.  \end{lemma}

We can now give our main theorem:

\begin{theorem} Let $y$ be a $d$-vector all of whose components are positive and let $x\in P(y)$.  Then the following are equivalent:

\begin{enumerate}
\item $x$ is an extreme point of $P(y)$.
\item $f_r(x)=f_r(y)$ for some $r\in \mathbb{R}$.
\item  Either $x$ is not trumped by $y$ or there exists some $d$-by-$d$ permutation matrix $P$ such that $x=Py$.
\end{enumerate}
\end{theorem}

\begin{proof}
  The equivalence of (2) and (3) are the results of Turgut and Klimesh, so we prove the equivalence of (1) and (2). Since $f_r(x)$ is either strictly convex or is the logarithm of a strictly convex function for any $r$, it follows that (2) implies (1).  We now show that (1) implies (3) by proving the contrapositive.  Our proof is by induction on $d$.  The base case of $d=2$ is immediate.  Now suppose our theorem holds for $d=n$ and let $x,y$ be $(n+1)$-tuples of positive numbers.  Suppose $x\prec_T y$ and $x$ is not a rearrangement of $y$.  If $x$ is an interior point of $T(y)\subseteq P(y)$, it can't be an extreme point of $P(y)$.  So suppose $x$ is a boundary point of $T(y)$, then by the previous lemma we must have $x_i=y_j$ for some $1\le i,j\le n+1$.  Let $\tilde{x}$ and $\tilde{y}$ be the $n$-tuples formed by removing $x_i$ and $y_j$ from $x$ and $y$ respectively.  Then $\tilde{x} \prec_T \tilde{y}$ and by our induction hypothesis, $\tilde{x}$ is not an extreme point of $P(\tilde{y})$.   Hence there exists $w,z\in P(\tilde{y})$, $w^{\uparrow}\neq z^{\uparrow}$ such that $\tilde{x}=\lambda w+(1-\lambda) z$ for some $\lambda\in (0,1)$. Then $x=\lambda (w_1,w_2,...,w_{i-1},x_{i},w_{i},...,w_n)+ (1-\lambda) (z_1,z_2,...,z_{i-1},x_{i},z_{i},...,z_n)$.  Since the latter two vectors are in $P(y)$, our result follows. \end{proof}

\section{Examples of Trumping}

There have been examples in the literature of vectors $x,y$ such that $x\preceq_p y$ but where $x$ is not majorized by $y$.  In \cite{Ben10}, using tools from \cite{BeJa00}, Bennett  gave an infinite family of such pairs.  We note that these examples are in fact examples of trumping and prove this by modifying Lemma 1 and Theorem 4 of \cite{BeJa00} slightly and then using these results in place of the originals. In particular, if we include strictness in one of the hypotheses of Lemma 1 of \cite{BeJa00}, we obtain strictness in the conclusion of the lemma (see Lemma \ref{BeJaL1}, below); similarly,  if we include strictness in the hypothesis of Theorem 4 of \cite{BeJa00}, we obtain strictness in the conclusion (see Theorem \ref{BeJaThm4}, below). The proofs of the modified versions of these two results follow from the original proofs with slight modifications, and so will be omitted.

\begin{lemma}\label{BeJaL1}
Suppose that $a<b<c<d$ and that $p, q, r$ are non-negative numbers. Suppose further that
\begin{eqnarray*}
b-a\leq d-c \textnormal{ and } p\lneq r\\
q(c-b)=p(b-a)+r(d-c).
\end{eqnarray*}
Let $g$ be a convex function on $[a, d]$ with $g(c)\gneq g(b)$. Then
\[
q\int_b^cg\lneq p\int_a^bg+r\int_c^dg.
\]
\end{lemma}

Let $M_n(f)$ denote the midpoint Riemann sum with $n$ subintervals for $\int_a^bf(t)dt$: If $[a,b]=[0,1]$, $M_n(f)=\frac1n \sum_{r=1}^nf(\frac{2r-1}{2n})$.  We can use Lemma \ref{BeJaL1} to prove a strengthened version of Theorem 4 of \cite{BeJa00} following the original proof using this new lemma.

\begin{theorem}\label{BeJaThm4}
Suppose that $f'$ is either convex or concave on $[a, b]$. If $f$ is strictly convex, then $M_n(f)$  strictly increases with $n$; if $f$ is strictly concave, then $M_n(f)$  strictly decreases with $n$.
\end{theorem}

\medskip
\medskip

\noindent{\textbf{Example }}
A system considered by Bennett (with non-strict inequalities) \cite{Ben10} is
\begin{eqnarray}\label{Bensys}
\frac{1^p}{1^{p+1}}< \frac{1^p+3^p}{2^{p+1}}< \frac{1^p+3^p+5^p}{3^{p+1}}< \cdots \frac{1^p+3^p+\cdots+(2n-1)^p}{n^{p+1}}< \cdots
\end{eqnarray}
for $p>1, p<0$ and reversed for $0<p<1$.

Bennett shows  that this system leads to power majorization by considering one inequality (say, the second one) and cross-multiplying, yielding $3^p+3^p+3^p+9^p+9^p+9^p\leq 2^p+2^p+6^p+6^p+10^p+10^p$ (for appropriate $p$). In other words, $x=(3, 3, 3, 9, 9, 9)\preceq_p (2, 2, 6, 6, 10, 10)=y$. He further mentions that the power majorizations found in this way are \emph{not majorizations}, save for the first inequality, and gives an example. To see this in general, we consider the $n$-th inequality of the system:
\[
\frac{1^p+3^p+\cdots+(2n-1)^p}{n^{p+1}}<\frac{1^p+3^p+\cdots+(2n+1)^p}{(n+1)^{p+1}},
\]
which yields, upon cross-multiplying, $n+1$ copies of $(n+1)^p+(3(n+1))^p+\cdots+\big((2n-1)(n+1)\big)^p$ being strictly less than (for appropriate $p$) $n$ copies of $n^p+(3n)^p+\cdots+(n(2n+1))^p$. These sums are already arranged in increasing order. We note that the first $n$ inequalities of majorization are met, but the $(n+1)$-th inequality is flipped:
\begin{eqnarray*}
n+1&>&n\\
2(n+1)&>&2n\\
&\vdots&\\
n(n+1)&>&n^2\\
(n+1)(n+1)&<& n^2+3n,
\end{eqnarray*}
for all $n>1$, thus we see explicitly where majorization fails.

Although the system from \cite{Ben10} is written with non-strict inequalities, we claim that the inequality is strict for $p\neq 0, 1$, hence equation (\ref{Bensys}) gives us an infinite number of trumping relations.

Indeed, if we rewrite the second inequality of equation (\ref{Bensys}), we have
\begin{eqnarray}\label{Bensys2}
 \frac{\left(\frac{1}{2}\right)^p+\left(\frac{3}{2}\right)^p}{2}< \frac{\left(\frac{1}{3}\right)^p+\left(\frac{3}{3}\right)^p+\left(\frac{5}{3}\right)^p}{3}.
\end{eqnarray}

Consider now something seemingly different: approximating the integral
 \[
\frac12\int_0^2x^p\, dx
\]
using midpoint Riemann sums. If we divide the interval $[0,2]$ into two equal intervals, our midpoints are $1/2$ and $3/2$; if we divide the interval $[0,2]$ into three equal intervals, our midpoints are $1/3, 3/3, 5/3$. Thus, asking whether the inequality is strict in (\ref{Bensys}) amounts to asking whether estimates of an integral of the function $f(x)=x^p$ using midpoint Riemann sums strictly improves as $n$ increases.  For instance, $M_2(x^p)<M_3(x^p)$ (where $[a,b]=[0,2]$) gives us $\frac{\left(\frac12\right)^p + \left(\frac32\right)^p}{2} <\frac{\left(\frac13\right)^p + \left(\frac33\right)^p + \left(\frac53\right)^p}{3}$ which is the second inequality in (\ref{Bensys}).  The full system of inequalities in (\ref{Bensys}) is $M_1(x^p)<M_2(x^p)<M_3(x^p)...$ for $p\in (-\infty,0)\cup (1,\infty)$ and with the inequalities reversed for $p\in (0,1)$.

We can show that these inequalities hold by using Theorem \ref{BeJaThm4}. We observe that $f(x)=x^p$ is strictly convex for $p>1, p<0$ and strictly concave for $0<p<1$, yielding $M_n(f)>M_m(f)$ for all $m<n$ when $p>1, p<0$ and the inequality reverses for $0<p<1$---precisely what is needed for strict power majorization. Thus the inequalities in system in (\ref{Bensys}) are strict, and in particular, $x=(3, 3, 3, 9, 9, 9)$ is strictly power majorized by $y=(2, 2, 6, 6, 10, 10)$.  To prove trumping, we use Observation \ref{fact} and note that one vector is composed of odd numbers and the other is composed of even numbers so neither product can be equal; the same idea works for any pair of $n$-tuples generated by the sequence of inequalities in (\ref{Bensys}) since every term in the numerator is odd and the denominators alternate between odd and even which means one of the tuples will consist entirely of odd numbers and the other entirely of evens.  This gives us an infinite sequence of pairs of vectors $(x,y)$ where $x$ is trumped but not majorized by $y$.

\medskip

\medskip

\noindent{\textbf{Example}}
Not every construction of the type considered in the previous example leads to non-trivial trumping (non-trivial in the sense that we do not also have majorization).

Indeed, if $p > 1$ or $p < 0$, we have the system of strict inequalities \cite{Ben05}
\[
\frac{1^p}{3^p} <\frac{1^p + 3^p}{5^p + 7^p} <\frac{1^p + 3^p + 5^p}{7^p + 9^p + 11^p} <\cdots <
\frac{1^p + \cdots+ (2n - 1)^p}{(2n + 1)^p + \cdots+ (4n - 1)^p} < \cdots
\]
The inequality reverses for $0<p<1$.

For sake of example, let us look at the second inequality,
\[
\frac{1^p + 3^p}{5^p + 7^p} <\frac{1^p + 3^p + 5^p}{7^p + 9^p + 11^p}.
\]
Cross-multiplying, we obtain
\[
7^p+9^p+11^p+21^p+27^p+33^p<5^p+7^p+15^p+21^p+25^p+35^p
\]
for $p > 1$ or $p < 0$, with reverse inequality for $0<p<1$. This is an example of strict power majorization; however, one can easily check that $(7, 9, 11, 21,27,33)$ is majorized by $(5, 7, 15, 21, 25, 35)$, and in that  sense this is a trivial example.
\medskip

\section*{Acknowledgements} The authors are grateful to Grahame Bennett for providing preprints of his work on majorization and inequalities. D.W.K. was supported by NSERC Discovery Grant 400160 and NSERC Discovery Accelerator Supplement 400233. R.P. was supported by NSERC Discovery Grant 400550. S.P. was supported by an NSERC Doctoral Scholarship.


\begin{thebibliography}{50}

\bibitem[1]{All88} G.\ D.\ Allen, \emph{Power Majorization and Majorization of Sequences}, Result.\ Math.\ \textbf{14} (1988), pp.\ 211--222.

\bibitem[2]{AuNe08} G.\ Aubrun and I.\ Nechita, \emph{Catalytic Majorization and $\ell_p$ Norms}, Comm.\ Math.\ Phys.\ \textbf{278}, no.\ 1 (2008), pp.\ 133--144.

\bibitem[3]{Ben86} G.\ Bennett, \emph{Power Majorization versus Majorization}, Anal.\ Math.\ \textbf{12} (1986), pp.\ 283-286.


\bibitem[4]{Ben05} G.\ Bennett, \emph{An Odd Inequality}, Problem 11139, Amer.\ Math.\ Monthly \textbf{112} (2005), p.\ 273.


\bibitem[5]{Ben10} G.\ Bennett, \emph{Some Forms of Majorization}, Houston J.\ Math.\ \textbf{36}, no.\ 4 (2010), pp.\ 1037--1066.

\bibitem[6]{Ben11} G.\ Bennett, \emph{A Problem of Ghorbani}, Anal.\ Math.\ \textbf{37} (2011), pp.\ 239-244.


\bibitem[7]{BeJa00} G.\ Bennett and G.\ Jameson, \emph{Monotonic Averages of Convex Functions}, J.\ Math.\ Anal.\ Appl.\ \textbf{252} (2000), pp.\ 410--430.

\bibitem[8]{Cla84} A.\ Clausing, \emph{A Problem Concerning Majorization}, General Inequalities 4 (W.\ Walter, Ed.), Birkh\"auser, Basel,
1984, p.\ 405.


\bibitem[9]{DaKl01} S.\ K.\  Daftuar and M.\ Klimesh, \emph{Mathematical structure of entanglement catalysis}, Phys.\ Rev.\ A (3) \textbf{64} no.\ 4 (2001), 042314.

\bibitem[10]{JoPl99}D.\ Jonathan and M.\ B.\ Plenio, \emph{Entanglement-Assisted Local Manipulation of Pure Quantum States}, Phys.\ Rev.\ Lett.\ \textbf{83} (1999), pp.\ 3566--3569.

\bibitem[11]{Kli04} M.\ Klimesh, \emph{Entropy Measures and Catalysis of Bipartite
Quantum State Transformations}, Proceedings 2004 IEEE International Symposium on Information Theory, Chicago, Il (USA), 2004.

\bibitem[12]{Kli07} M.\ Klimesh, \emph{Inequalities that Collectively Completely Characterize
the Catalytic Majorization Relation}, 2007. Available at: http://arxiv.org/abs/0709.3680v1

\bibitem[13]{LoPr01} H.-K.\ Lo and S.\ Popescu, \emph{Concentrating Entanglement by Local Actions: Beyond Mean Values}, Phys.\ Rev.\ A, \textbf{63} (2001), 022301.

\bibitem[14]{MOA11} A.\ W.\ Marshall, I.\ Olkin, and B.\ C.\ Arnold, \emph{Inequalities: Theory of Majorization and its Applications}, 2nd ed. Springer: New York, 2011.

\bibitem[15]{Nie99} M.\ Nielsen, \emph{Conditions for a Class of Entanglement Transformations}, Phys.\ Rev.\ Lett.\ \textbf{83}, no.\ 2 (1999), pp.\ 436--439.
\bibitem[16]{Tur07} S.\ Turgut, \emph{Catalytic Transformations for Bipartite Pure States}, J.\ Phys.\ A: Math.\ Theor.\ \textbf{40} (2007), pp.\ 12185--12212.

\end{thebibliography}
\end{document}